\documentclass[a4paper,11pt]{amsart}

\usepackage{amsmath,amsthm,amssymb}
\usepackage{comment}
\usepackage{setspace}
\usepackage{color} 
\usepackage{pdfsync}
\newtheorem{theorem}{Theorem}
\newtheorem{corollary}[theorem]{Corollary}
\newtheorem{lemma}[theorem]{Lemma}
{\theoremstyle{remark}
\newtheorem{remark}{Remark}

}
\newcommand{\RR}{\mathbb{R}}
\newcommand{\NN}{\mathbb{N}}

\DeclareMathOperator{\trace}{trace}
\DeclareMathOperator*{\osc}{osc}

\DeclareMathOperator{\diam}{diam}

\newcommand{\tx}[1]{{\text{\rm #1}}}
\newcommand{\pd}[2]{\dfrac{\partial#1}{\partial#2}}
\newcommand{\pdd}[3]{\dfrac{\partial^2#1}{\partial#2 \partial#3}}

\numberwithin{equation}{section}
\title[Gradient estimates for parabolic equations]{Time-interior gradient estimates for quasilinear parabolic equations}
\author{Ben Andrews}
\thanks{Research partially supported by a Discovery Grant of the Australian Research Council}

\address{Centre for Mathematics and its Applications,  Australian National University, A.C.T. 0200, Australia}
\email{Ben.Andrews@maths.anu.edu.au}
\author{Julie Clutterbuck}
\address{Centre for Mathematics and its Applications,  Australian National University, A.C.T. 0200, Australia}
\email{Julie.Clutterbuck@maths.anu.edu.au}
 
\begin{document}
  \definecolor{draftgrey}{gray}{0.4}
  \begin{abstract}{
Bounded smooth solutions of the Dirichlet and Neumann problems for a wide variety of quasilinear parabolic equations, including graphical anisotropic mean curvature flows, have gradient  bounded in terms of oscillation and elapsed time.  } \end{abstract}

\maketitle

\section{Introduction}\label{sec:Introduction}

Our aim in this paper is to prove time-interior gradient estimates (more precisely, estimates on the gradient for positive times which do not depend on the initial gradient, but only on the oscillation of the initial data) for solutions of quasilinear parabolic equations with gradient-dependent coefficients, under the weakest possible assumptions on the coefficients.  We are motivated by geometrically natural equations such as the graphical mean curvature flow and its anisotropic analogues, which turn out to be borderline cases for such estimates.  

The graphical mean curvature flow is the following
degenerate quasilinear parabolic partial
differential equation for a function $u$:
\begin{equation}\label{eq:GMCF}
\frac{\partial u}{\partial t}=\sqrt{1+|Du|^2}
D_i\left(\frac{D_iu}{\sqrt{1+|Du|^2}}\right)
= \left(\delta_{ij}-\frac{D_iuD_ju}{1+|Du|^2}\right)\!D_iD_ju.
\end{equation}
The graph of a solution to this equation moves in such a way that the normal 
component of its velocity is equal to the mean curvature.  This is geometrically a very natural deformation, and it arises in particular as the steepest descent flow for the area functional.

Stationary solutions of the graphical mean curvature flow equation are minimal hypersurfaces:  The equation defining these has been studied in
great detail, and has a well-developed regularity theory.  Note that the stationary equation can be written in divergence form:
\begin{equation}\label{eq:MinSurf}
D_i\left(\frac{D_iu}{\sqrt{1+|Du|^2}}\right)=0.
\end{equation}

The regularity theory for the graphical mean curvature flow is more difficult than that for the minimal surface equation, partly because of its non-divergence form.  Ecker and Huisken \cite{eh:interior} 
proved estimates on the gradient of solutions which were interior in space but not time
--- the estimates depended on an initial gradient bound.
Fully interior bounds on all higher
derivatives (given the gradient bound) are also proved in \cite{eh:interior}.

Evans and Spruck \cite[Theorem 5.2]{es:gradest} proved interior gradient estimates for solutions of graphical mean curvature flow, by applying maximum principle arguments to quantities involving the gradient and the height.  Their result is not well known, and has been rediscovered at least once using essentially the same argument \cite{cm:gradest}. 

A more general evolution equation which is of considerable interest is the anisotropic mean curvature flow, which is analogous to the usual mean curvature flow but explicitly incorporates some anisotropic dependence on the normal direction.  Such flows arise naturally as gradient flows of area-like functionals, and appear in models of crystal growth \cite{cht} and other physical phenomena involving moving interfaces \cite{anggur1}, \cite{anggur2}.  We will follow the treatment given in \cite{andrews:anisotropic}.  In the graphical setting these equations have 
the following form:
\begin{equation}\label{eq:AMCF}
\frac{\partial u}{\partial t} = a^{ij}(Du)D_{i}D_{j}u
\end{equation}
with
\begin{equation}\label{eq:AMCFcoeff}
a^{ij}(p) = m(p)F(p)\frac{\partial^{2}F}{\partial p_{i}\partial p_{j}}.
\end{equation}
Here $F(p) = \bar F(p,-1)$, where $\bar F$ is a function on $(\RR^{n+1})^{*}$ which is
\begin{itemize}
\item positively homogeneous of degree one:\ $\bar F(\lambda z)=\lambda\bar F(z)$ for $\lambda>0$;
\item positive:  $\bar F(z)>0$ for each $z\in(\RR^{n+1})^{*}\backslash\{0\}$;
\item smooth on $(\RR^{n+1})^{*}\backslash\{0\}$;
\item strictly convex in non-radial directions:  For any $p\in(\RR^{n+1})^{*}\backslash\{0\}$ and any $v$ which is not a multiple of $p$, $D^{2}\bar F\big|_{p}(v,v)>0$. 
\end{itemize}
The function $m$ is called the mobility function, and is given by $m(p)=\bar m(p,-1)$ where $\bar m$ is a function on $(\RR^{n+1})^{*}\backslash\{0\}$ which is
\begin{itemize}
\item positively homogeneous of degree zero: $\bar m(\lambda z) = \bar m(z)$ for $\lambda>0$;
\item positive:  $\bar m(z)>0$ for each $z\in(\RR^{n+1})^{*}\backslash\{0\}$;
\item smooth on $(\RR^{n+1})^{*}\backslash\{0\}$.
\end{itemize}
 
The methods of \cite{es:gradest} and \cite{cm:gradest} do not extend easily to this more general setting.   The second author has recently shown \cite{julie:anisotropic} that the methods of these papers can be extended to give fully interior gradient estimates for anisotropic flows under additional assumptions on smallness and symmetry of the anisotropy.  Such estimates are not known without these restrictions, except in the one-dimensional case \cite[Theorem 11.18]{lieb:book}, \cite{NagaseTonegawa}, \cite{ac:gradest1D}.

The aim of this paper is to introduce a new method of proof which yields gradient estimates
for spatially periodic solutions of graphical mean curvature flow, and for Dirichlet and Neumann 
initial boundary value problems.  The estimates are not local in space, but they do apply to anisotropic mean curvature flows as well as the isotropic mean curvature flow.  The methods also give sharp gradient estimates for solutions of many other equations of interest including the $p$-Laplacian heat equations, and more generally quasilinear equations with coefficients depending on the gradient.

In our previous paper \cite{ac:gradest1D} we adapted a method of Kruzhkov 
\cite{kruzhkov:quasilinear} to prove estimates for parabolic
equations in one space variable.    The main result of that paper was that the 
modulus of continuity of a (periodic) solution of an equation of the form
\begin{equation}\label{eq:1Dflow}
\frac{\partial u}{\partial t} = \alpha(u')u''
\end{equation}
is a subsolution of the same equation.  More precisely, if $u$ is an $L$-periodic regular solution of 
\eqref{eq:1Dflow}
(by which we mean that the second spatial derivatives and first time derivatives are continuous), and
$\psi$ is a concave positive function on $(0,L/2)$ such that
\begin{equation}\label{eq:modofcont}
-2\psi\left(\frac{L+x-y}{2}\right)\leq u(y,0)-u(x,0) \leq 2\psi\left(\frac{y-x}{2}\right)
\end{equation}
for every $x,y$ with $0<y-x<L$ (we say $\psi$ is a \emph{modulus of continuity} for the initial data $u(.,0)$), then the (viscosity) solution of \eqref{eq:1Dflow} on $[0,L/2]$ with Dirichlet boundary condition and initial condition $\psi$ gives a modulus of continuity for $u$ at any positive time.  That is, there exists a unique minimal $\psi_+:\ [0,L/2]\times[0,\infty)$ which is continuous, non-negative, concave for each $t\geq 0$, with $\psi_+(x,0)=\psi(x)$ for each $x$, and is regular and satisfies \eqref{eq:1Dflow} on $(0,L/2)\times(0,\infty)$.  We prove that $\psi_+(.,t)$ is a modulus of continuity for $u(.,t)$ for every $t\geq 0$ (in the sense of \eqref{eq:modofcont}).  This estimate is sharp:  There is no smaller modulus of continuity which holds for arbitrary regular solutions with initial modulus of continuity $\psi$.
We also gave a simple necessary and sufficient condition on the coefficient $\alpha$ for the modulus of continuity $\psi_+$ to have bounded gradient for positive times for arbitrary bounded $\psi$:  This is true precisely when $\int_0^\infty s\alpha(s)\,ds$ and $\int_{-\infty}^0s\alpha(s)\,ds$ both diverge.   Thus this condition is also a necessary and sufficient condition for the existence of bounds of the form
$$
|u'(x,t)|\leq C(\|u\|_\infty,t)
$$
for arbitrary regular $L$-periodic solutions of \eqref{eq:1Dflow}.

The main work of the present paper is in the
non-trivial extension of the methods of \cite{ac:gradest1D} to higher-dimensions:  In section \ref{sec:mcf} we treat isotropic flows including the graphical
mean curvature flow, and obtain a necessary and sufficient criterion for time-interior gradient estimates to hold for such equations.  In section \ref{sec:general} the argument is extended to more general (non-isotropic) equations, including anisotropic mean curvature flows.  Boundary value problems are treated in Section \ref{sec:bvp}.
We conclude in Section \ref{sec:app} by giving some applications of the
estimates to solve initial and boundary-value problems with initial data which are merely continuous.

\section{Periodic mean curvature flow and other isotropic equations}\label{sec:mcf}

The argument of Kruzhkov proceeds by doubling the number of spatial variables, introducing a new function $w(x,y,t)=u(y,t)-u(x,t)$.  This is easily seen to satisfy a parabolic equation in two spatial variables, and a gradient estimate for $u$ can be deduced from a boundary gradient estimate for $w$ along the diagonal line $\{y=x\}$.  Such a boundary gradient
estimate can be obtained in a straightforward manner by constructing
a barrier.  
The extension of these methods to the
higher-dimensional mean curvature flow
seems implausible, since then we have a parabolic equation in $2n$
variables while the diagonal
$\{y^i=x^i,\ i=1,\dots, n\}$ is $n$-dimensional.   One does not normally expect to be
able to produce a barrier for such a boundary.  Nevertheless, we will show that the method works, at least for a wide class of equations where the coefficients depend only on the gradient.  The key is to make as much use as possible of the full $2n\times 2n$ matrix of second
derivatives, including the mixed derivatives $\frac{\partial^{2}}{\partial x^{i}\partial y^{j}}$.

We will illustrate this first in a particularly simple and natural case:  The graphical mean curvature flow
\begin{equation}
\dfrac{\partial u}{\partial t} = \left(\delta_{ij}-\frac{D_iu\,D_ju}{
1+|Du|^2}\right)D_iD_ju.\label{eq:mcf}
\end{equation}
We show the moduli of continuity of bounded solutions of \eqref{eq:mcf} are controlled by a particular solution $\varphi$ of the graphical curve-shortening flow:  
\begin{equation}\label{eq:csf}
\frac{\partial\varphi}{\partial t} = \frac{\varphi''}{1+(\varphi')^2}.
\end{equation}
This solution $\varphi$ is smooth on $[0,\infty)\times(0,\infty)$ with $\varphi(0,t)=0$ for $t>0$, and $\varphi(x,t)\to\frac12$ as $t\to 0$ for any $x>0$ or $x\to \infty$ for any $t>0$.  Note that for any $M>0$, the function $\phi(x,t)=M\varphi\left(\frac{x}{M},\frac{t}{M^2}\right)$ is again a solution of \eqref{eq:csf}.

Fix a linearly independent set $\Gamma=\{v_{1},\dots,v_{n}\}$ in $\RR^{n}$.  A function $f$ on $\RR^{n}$ is called $\Gamma$-periodic if for each $x\in\RR^{n}$ and $1\leq i\leq n$,  $f(x+v_{i})=f(x)$.

\begin{theorem}\label{Theorem 2}
Let $u: {\RR}^n\times [0,T)\to{\RR}$ be a smooth solution to equation \eqref{eq:mcf}
with an oscillation bound $|u(x,0)-u(y,0)|\leq M$, and with $u$ 
$\Gamma$-periodic.
Then for all $x$ and $y$ in ${\RR}^n$ and $t>0$,
$$
|u(y,t)-u(x,t)|\leq 2M\varphi\left(\frac{|y-x|}{2M},\frac{t}{M^2}\right).
$$
In particular $1+|Du(x,t)|^2\leq \exp\left(\frac{2M^2}{t}\right)$
for all $x\in\RR^n$ and all $t>0$.
\end{theorem}

\begin{proof}
The proof is similar to that of Theorem 1 in \cite{ac:gradest1D}:
Let $\varepsilon>0$, and define 
$$
Z(x,y,t)=u(y,t)-u(x,t)-2\phi\left(\frac{|y-x|}{2},t\right)-\varepsilon(1+t)
$$ 
on $\{y\neq x\}\times(0,\infty)$, where $\phi(\xi,t)=M\varphi(\xi/M,t/M^2)$.  $Z$ is strictly negative near the boundary $\{y=x\}$, and everywhere for small times.  Since $\varphi$ is increasing,  $Z$ is increased by choosing $y$ so that $|y-x|\leq |y+j-x|$ for all lattice shifts $j$.  Thus we need only consider $x$ and $y$ in a bounded region, so the maximum of $Z$ is attained for each $t$.  If $Z$ is not negative for all $(x,y,t)\in\RR^n\times\RR^n\times[0,\infty)$, then there exists $t_0>0$ and $y_0\neq x_0$ in $\RR^n$ such that $Z(x_0,y_0,t_0)=\sup\{Z(x,y,t):\ x,y\in\RR^n,\ 0\leq t\leq t_0\}=0$.
At this point we have the first order conditions:
\begin{equation}\label{eq:firstorder}\begin{aligned}
0&=\pd{Z }{x^i}=
-\pd{u(x,t)}{x^i}+{\phi'}\frac{y^i-x^i}{|y-x|}\\
0&=\pd{Z}{y^i}=
\pd{u(y,t)}{y^i}-{\phi'}\frac{y^i-x^i}{|y-x|}.
\end{aligned}\end{equation}

These become simpler if we choose coordinates such that
$e_1=\frac{y-x}{|y-x|}$:  Then 
\begin{equation} \label{first derivative 1}  \begin{split}
D_1u(x,t)&=\phi' \\
D_ju(x,t)&=0
\text{ for $j=2,\dots,n$ } \\
D_1u(y,t)&=\phi' \\
D_ju(y,t)&=0
\text{ for $j=2,\dots,n$. }
\end{split}\end{equation}

The second derivatives of $Z$ are as follows:
\begin{equation}
\begin{split}
\frac{\partial^2Z}{\partial x^i\partial x^j}
&\!=\!-
\pdd{u(x,t)}{x^i}{x^j}
\!-\!\frac{\phi''(y^i-x^i)(y^j-x^j)}{2|y-x|^2}
\!-\frac{\phi'}{|y-x|}\!\!\left(\!\delta_{ij}\!-\!\frac{(y^i-x^i)(y^j-x^j)}{|y-x|^2}
\!\right)\cr
\frac{\partial^2Z}{\partial y^i\partial y^j}
&\!=\!
\pdd{u(y,t)}{y^i}{y^j}
-\frac{\phi''(y^i-x^i)(y^j-x^j)}{2|y-x|^2}
-\frac{\phi'}{|y-x|}\!\!\left(\!\delta_{ij}\!-\!\frac{(y^i-x^i)(y^j-x^j)}{|y-x|^2}
\right)\cr
\frac{\partial^2Z}{\partial x^i\partial y^j}
&\!=\!\frac{\phi''(y^i-x^i)(y^j-x^j)}{2|y-x|^2}
+\frac{\phi'}{|y-x|}\left(\delta_{ij}-\frac{(y^i-x^i)(y^j-x^j)}{|y-x|^2}
\right).
\end{split}\notag\end{equation}
Choosing coordinates as before, {\allowdisplaybreaks this simplifies to give
\begin{align}\label{second derivatives of Z} \begin{split}
Z_{x^1x^1}&=-D_1D_1u(x,t)-\frac12\phi''\\
Z_{x^ix^i}&=-D_iD_iu(x,t)-\frac{\phi'}{|y-x|}\quad\text{\rm if $i>1$}\\
Z_{x^ix^j}&=-D_iD_ju(x,t)\quad\text{\rm if $i\neq j$}\\
Z_{y^1y^1}&=D_1D_1u(y,t)-\frac12\phi''\\
Z_{y^iy^i}&=D_iD_iu(y,t)-\frac{\phi'}{|y-x|}\quad\text{\rm if $i>1$}\\
Z_{y^iy^j}&=D_iD_ju(y,t)\quad\text{\rm if $i\neq j$}\\
Z_{x^1y^1}&=\frac12\phi''\\
Z_{x^iy^i}&=\frac{\phi'}{|y-x|}\quad\text{\rm if $i>1$}\\
Z_{x^ix^j}&=0\quad\text{\rm if $i\neq j$.} \end{split}
\end{align}
The $2n\times 2n$ matrix }of second derivatives of $Z$ is
negative semi-definite:
$$
0\geq \left[D^2Z\right]=
\bmatrix
Z_{x^1x^1}&\dots  &Z_{x^1x^n}&Z_{x^1y^1}&\dots &Z_{x^1y^n}\cr
\vdots    &\ddots &\vdots    &\vdots    &\ddots&\vdots\cr
Z_{x^nx^1}&\dots  &Z_{x^nx^n}&Z_{x^ny^1}&\dots &Z_{x^ny^n}\cr
Z_{y^1x^1}&\dots  &Z_{y^1x^n}&Z_{y^1y^1}&\dots &Z_{y^1y^n}\cr
\vdots    &\ddots &\vdots    &\vdots    &\ddots&\vdots\cr
Z_{y^nx^1}&\dots  &Z_{y^nx^n}&Z_{y^ny^1}&\dots &Z_{y^ny^n}\cr
\endbmatrix. 
$$

At the maximum point, the time derivative of $Z$ is given by
\begin{align*}
\frac{\partial Z}{\partial t} &= 
\!\left(\!\delta_{ij}\!-\!\frac{D_iu(y)D_ju(y)}{
1+|Du(y)|^2}\!\right)\!\!D_iD_ju(y,t)\\
&\quad\null
-\!\left(\!\delta_{ij}\!-\!\frac{D_iu(x)D_ju(x)}{
1+|Du(x)|^2}\!\right)\!\!D_iD_ju(x,t) -2\frac{\partial\phi}{\partial t}-\varepsilon\cr 
&<\!\left(\!\delta_{ij}\!-\!\frac{D_iu(y)D_ju(y)}{
1+|Du(y)|^2}\!\right)\!\!Z_{y^iy^j}
+\!\left(\!\delta_{ij}\!-\!\frac{D_iu(x)D_ju(x)}{
1+|Du(x)|^2}\!\right)\!\!Z_{x^ix^j}+2c^{ij}Z_{x^iy^j}\cr
&\quad\null+\frac{\phi''}{ 1+(\phi')^2}+2(n-1)\frac{\phi'}{|y-x|}
-{c^{11}}\phi''-2\sum_{i=2}^nc^{ii}\frac{\phi'}{|y-x|}-
2\frac{\partial\phi}{\partial t}.
\end{align*}
If the coefficient matrix of the second derivatives of $Z$ is
positive semi-definite,
so that
$$
0\leq 
\bmatrix
\frac{1}{ 1+(\phi')^2}&0&\dots&0&c^{11}&c^{12}&\dots&c^{1n}\cr
0                   &1&\dots&0&c^{21}&c^{22}&\dots&c^{2n}\cr
\vdots              &\vdots&\ddots&\vdots&\vdots&\vdots&\ddots&\vdots\cr
0                   &0&\dots &1&c^{n1}&c^{n2}&\dots&c^{nn}\cr
c^{11}&c^{21}&\dots&c^{n1}&\frac{1}{ 1+(\phi')^2}&0&\dots&0\cr
c^{12}&c^{22}&\dots&c^{n2}&0&1&\dots&0\cr
\vdots&\vdots&\ddots&\vdots&\vdots&\vdots&\ddots&\vdots\cr
c^{1n}&c^{2n}&\dots&c^{nn}&0&0&\dots&1\cr
\endbmatrix
$$
then the terms in the first line of the expression for $\partial
Z/\partial t$ are non-positive.  In particular, we can choose
\begin{align*}
c^{11}&=-\frac{1}{ 1+(\phi')^2}\cr
c^{ii}&=1\quad\text{\rm if $i>1$}\cr
c^{ij}&=0\quad\text{\rm if $i\neq j$}
\end{align*}
yielding the following inequality at the maximum point:
$$
\frac{\partial Z}{\partial t}< 2\left(\frac{\phi''}{ 1+(\phi')^2}
-\frac{\partial\phi}{\partial t}\right)=0.
$$
This contradiction proves that $Z(x,y,t)<0$ for all $x,y\in\RR^n$ and all $t\geq 0$.  Letting $\varepsilon\to 0$ we obtain the result of the Theorem.  The explicit estimate follows exactly as in Theorem 7 (and Corollary 8) of \cite{ac:gradest1D}.
\end{proof}

This result is sharp:  Initial data close to a square-wave function of one of the variables will give equality in the limit of lattices with large period.

The same argument as above gives sharp gradient estimates for arbitrary `isotropic' quasilinear equations, to give sharp control on the modulus of continuity of solutions at positive times in terms of their initial modulus of continuity.  In analogy with \eqref{eq:modofcont}, we say a positive concave function on $(0,\infty)$ is a \emph{modulus of continuity} for a function $v$ on $\RR^n$ if for all $y\neq x$ in $\RR^n$,
\begin{equation}\label{eq:modofcontRn}
|v(y)-v(x)|\leq 2\psi\left(\frac{|y-x|}{2}\right).
\end{equation}

\begin{theorem}\label{thm: nD isotropic}
Let $u: \RR^{n}\times [0,T)\to\RR$ be a regular solution to the equation
\begin{equation}\label{eq:isotropic}
\frac{\partial u}{\partial t} = \left(\alpha(|Du|,t)\frac{D_{i}uD_{j}u}{|Du|^{2}}
+\beta(|Du|,t)\left(\delta_{ij}-\frac{D_{i}uD_{j}u}{|Du|^{2}}\right)\right)D_{i}D_{j}u
\end{equation}
where $\alpha$ and $\beta$ are non-negative functions on $[0,\infty)$, and $u$ is $\Gamma$-periodic. Suppose $\psi$ is a modulus of continuity for $u(.,0)$.  Let $\varphi$ be non-negative, increasing in $x$, regular on $[0,\infty)\times[0,\infty)$, and satisfy 
\begin{equation}\label{eq:1Dreduction}
\frac{\partial \varphi}{\partial t} \geq \alpha(\varphi',t)\varphi''
\end{equation}
with $\varphi(z,0)\geq \psi(z)$ for $z>0$.  Then for all $t>0$,
$\varphi(.,t)$ is a modulus of continuity for $u(.,t)$.
\end{theorem}

Note that $\varphi$ need only be defined on $[0,L]\times[0,\infty)$, for any $L$ greater than the diameter of a unit cell of the lattice.

Particular cases of interest include the $p$-Laplacian heat flows
$$
\frac{\partial u}{\partial t} = D_{i}\left(|Du|^{p-2}D_{i}u\right)
$$
for $p>1$.  In these examples $\alpha(q)= (p-1)|q|^{p-2}$, so the one-dimensional equation \eqref{eq:1Dreduction} is simply the one-dimensional $p$-Laplacian heat flow, and a useful solution $\varphi$ can be constructed explicitly:    Define
\begin{equation*}
F_{p}(\xi) = \left\{\begin{aligned}
&\int_{0}^{\xi}\left(1-s^{2}\right)_{+}^\frac{1}{ p-2}\,ds,&p>2;\\
&\int_{0}^{\xi}\exp\left(-s^{2}\right)\,ds,& p=2;\\
&\int_{0}^{\xi}\left(1+s^{2}\right)^{-\frac{1}{ 2-p}}\,ds,& 1<p<2.
\end{aligned}\right.
\end{equation*}
Then set $F_{p}(\infty):=\lim_{\xi\to\infty}F_{p}(\xi)$,
and define
\begin{equation*}
R_{p}=\left\{\begin{aligned}
&\left(\frac{2p(p-1)}{ p-2}\right)^\frac{1}{ p}\left(2F_{p}(\infty)\right)^{-\frac{p-2}{ p}},&p>2;\\
&2,& p=2;\\
&\left(\frac{2-p}{ 2p(p-1)}\right)^{-\frac{1}{ p}}\left(2F_{p}(\infty)\right)^{\frac{2-p}{ p}},& 1<p<2;\\
\end{aligned}\right.
\end{equation*}
The required solutions are then
\begin{equation*}
\varphi(z,t) =
\frac{1}{ 2F_{p}(\infty)}F_{p}\left(\frac{z}{ t^{\frac{1}{ p}}R_{p}}\right).
\end{equation*}

This implies the sharp gradient bounds for periodic solutions of the $p$-Laplacian heat equation in any dimension:
$$
|Du(x,t)|\leq \frac{1}{ 2R_{p}F_{p}(\infty)} M^{\frac{2}{ p}}t^{-\frac{1}{ p}},
$$
where $M=\osc u=\sup u-\inf u$.
Note that for $p>2$ weak solutions must be used.  However it is known \cite{dibenedetto}, \cite{chendiben1}, \cite{chendiben2} that these solutions are $C^{1,\alpha}$, and hence smooth away from points where the gradient vanishes.  This is sufficient for the above proof to apply.  

\section{The general periodic case}\label{sec:general}

%% more remarks on surprising result eg amcf
In this section we show that the methods of Section \ref{sec:mcf} can be extended to a wide range of singular or degenerate quasilinear parabolic equations.    The main requirement for application of the method is sufficient control on the degeneracy of the equations as the gradient becomes large.  In particular, it is sufficient to require that the smallest eigenvalue of the coefficient matrix $a^{ij}$ is at least comparable to $1/|Du|^{2}$ as $|Du|$ becomes large.  This condition is satisfied by any anisotropic mean curvature flow of the form given in Equation \eqref{eq:AMCF}--\eqref{eq:AMCFcoeff}  (see Lemma \ref{lem:amcf} below).
Consider the
 evolution equation 
\begin{equation}
u_t=a^{ij}(Du,t)D_{i}D_{j}u+b(Du,t)\label{evolution equation}, 
\end{equation} 
where $A(p,t)=[a^{ij}(p,t)]$ is positive semi-definite.

Suppose that there exists a continuous $\alpha:\RR^+\times[0,T]\rightarrow\RR_+$ with
\begin{equation}
0<\alpha(R,t)\le R^2 \inf_{|p|=R, (v\cdot p)\not= 0}
\frac{ v^T A(p,t)v}{(v\cdot p)^2}.\label{defn of alpha}
\end{equation}  
Let $\psi:\ [0,\infty)\to\RR$ be non-negative and concave,
and suppose that $\varphi:\ [0,\infty)\times[0,\infty)\to\RR_+$ is regular, non-decreasing and non-negative and satisfies 
\begin{equation} \label{conditions and equation for varphi}
\varphi_t\geq \alpha(|\varphi'|,t)\varphi'' \end{equation}
with $\varphi(z,0)\geq\psi(z)$ for $z>0$.
The result is as follows:

\begin{theorem}
\label{periodic higherdimensional theorem}
Let $u$ be a regular $\Gamma$-periodic solution to \eqref{evolution equation}.  If $\psi$ is a modulus of continuity for $u(.,0)$ then $\varphi(..,t)$ is a modulus of continuity for $u(.,t)$, for every $t>0$.
\end{theorem}

\begin{corollary} \label{periodic higherdimensional corollary}
If there are positive constants $A_0$ and $P$ so that 
\begin{equation} 
\alpha(|p|,t)|p|^2\ge A_0 \text{ for $|p|\ge P $ and all }t\label{degeneracy for alpha}, 
\end{equation}
 then any regular $\Gamma$-periodic solution $u$ with $\osc u\leq M$ satisfies
\begin{equation*}
|Du(x,t)|\le P\exp\left(1+\frac{M^2}{A_0t}\right)
\end{equation*}
for all $x\in\RR^n$ and $t>0$.\end{corollary}

\begin{proof}[Proof of Theorem \ref{periodic higherdimensional theorem}] 
We begin as in the proof of Theorem \ref{Theorem 2}:  Define  $Z(x,y,t):=u(y,t)-u(x,t)-2\varphi(|y-x|/2,t)-\varepsilon(1+t)$ on the domain $\{y\not=x\}$.  The assumptions on $\varphi$ guarantee that $Z<0$ near the boundary $\{y=x\}$ for any $t$, and that $Z<0$ everywhere for small $t$. Let $t_0$ be the first time where there exist points $x_0\neq y_0$ where $Z(x_0,y_0,t_0)=0$ (note that since $\varphi$ is non-decreasing, 
$|y_0-x_0|\leq |y_0-x_0+j|$ for every lattice point $j$, and so $x_0$ and $y_0$ can be chosen in a compact region of $\RR^n$). The first order condition  \eqref{eq:firstorder} is satisfied at $(x_0,y_0,t_0)$, while the matrix of second derivatives $[D^2Z]$ is negative semi-definite with entries given by \eqref{second derivatives of Z}.

At the maximum point, 
\begin{align}\label{eq:dZdt}
0\leq \frac{\partial Z}{\partial t}&= u_t(y,t)-u_t(x,t)-2\frac{\partial\phi}{\partial t}-\varepsilon\notag\\ 
&< {a}^{ij}(D_yu,t)D_iD_ju(y,t)+b(D_yu,t)\notag\\
&\phantom{=}
-{a}^{ij}(D_xu,t)D_iD_ju(x,t)-b(D_xu,t)
        -2\frac{\partial\phi}{\partial t}\notag\\
& =  {a}^{ij}(\phi'e_1,t) Z_{y^iy^j} +{a}^{ij}(\phi'e_1,t)Z_{x^ix^j}  
+2c^{ij}Z_{x^iy^j}   -2\frac{\partial\phi}{\partial t} \notag\\
&\phantom{=} 
+\phi''a^{11}(\phi'e_1,t)+2\frac{\phi'}{|y-x|}\sum_{i=2}^n{a}^{ii}(\phi'e_1,t) 
-c^{11}\phi''-2\frac{\phi'}{|y-x|}\sum_{i=2}^nc^{ii}\notag\\
&=\text{trace}
        \left(\begin{bmatrix} {A}(\phi'e_1) &  {C} \\  {C}^T &  {A}(\phi'e_1) 
\end{bmatrix} D^2Z \right) -2\frac{\partial\phi}{\partial t}\\
&\phantom{=}+ \left( {a}^{11}- {c}^{11}\right)\phi''  
+2\frac{\phi'}{|y-x|}\sum_{i=2 }^n\left( {a}^{ii}- {c}^{ii}\right),\notag
\end{align}
where in the third step, we have added and subtracted  $2\trace{\left(CZ_{xy}\right)}$, for some $n\times n$ matrix $C$. If we choose $C$ so that  the $2n\times 2n$ matrix 
\begin{equation*}
A'= \begin{bmatrix}  {A} &  {C} \\  {C}^T &  {A} \end{bmatrix}
\end{equation*}
is positive semi-definite,  then the trace term in \eqref{eq:dZdt} will be non-positive.  

Next, in order to make the coefficient of $\phi'/|y-x|$ zero, we require $c^{ij}=a^{ij}(\phi'e_1)$ for $(i,j)\not=(1,1)$.  

Finally, we choose $c^{11}$ to maximise the coefficient of $\phi''$.   The condition $A'\ge 0$ is equivalent to $0\le 2 v^T A(\phi'e_1) v- \left(a^{11}-c^{11}\right)(v^1)^2$, and so we choose $c^{11}=a^{11}(\phi'e_1)-2\alpha(|\phi'|)$ where $\alpha$ is defined by \eqref{defn of alpha}.  This gives the following inequality at $(x_0,y_0,t_0)$:
\begin{align*} 
0\leq \frac{\partial Z}{\partial t}&< 2\left(\alpha\left(|\phi'|\right) \phi''-\frac{\partial\phi}{\partial t}\right)\le 0,
\end{align*}
a contradiction which proves $Z(x,y,t)<0$ for all $(x,y)$ and $t\geq 0$. 
The estimate for $|u(y,t)-u(x,t)|$ follows after taking $\varepsilon\to 0$.
\end{proof}

\begin{proof}[Proof of Corollary \ref{periodic higherdimensional corollary}]
This follows using translating supersolutions as barriers exactly as in Theorem 7 of \cite{ac:gradest1D} (with $z=P$ and $b(z)=M$).
\end{proof}

We can apply Theorem \ref{periodic higherdimensional theorem} to anisotropic mean curvature flows:

\begin{lemma}\label{lem:amcf} Let $a^{ij}(p)$ be as given in Equation \eqref{eq:AMCFcoeff}.  Then there exists a  positive constant $A$ such that 
\begin{equation*}
a^{ij}(p)v^iv^j\ge \frac{A|v|^{2}}{1+|p|^{2}}
\end{equation*}
for all $p\in\RR^n$.  
\end{lemma}

\begin{proof}  
Recall that $F(p)=\bar F(p,-1)$, with $\bar F$ positively homogeneous of degree one and strictly convex in non-radial directions (that is, for any $q\in\RR^{n+1}\backslash\{0\}$ and any $v\in\RR^{n+1}$ which is not a multiple of $q$, 
$D^{2}\bar F\big|_{q}(v,v)>0$).  By compactness of the sphere, there exists a positive constant $A_{1}$ such that whenever $|q|=1$ and $v\in T_{q}S^{n}$, 
\begin{equation}\label{eq:D2Fsphere}
D^{2}\bar F\big|_{q}(v,v)\geq A_{1}|v|^{2}.
\end{equation}
Also, since $\bar F$ and $\bar m$ are positive away from the origin, there exists a constant $A_{2}$ such that $\bar m(q)\bar F(q)\geq A_{2}$ for $|q|=1$.

Observe that by homogeneity of $\bar F$, we have several useful facts:  First, $D\bar F$ is positively homogeneous of degree zero:
\begin{align}\label{eq:DFhom0}
D\bar F\big|_{\lambda q}(v) &= \frac{d}{ ds}\bar F(\lambda q+sv)\big|_{s=0}\notag\\
&=\lambda \frac{d}{ ds}\bar F\left(q+\frac{s}{ \lambda} v\right)\big|_{s=0}\notag\\
&=D\bar F\big|_{q}(v).
\end{align}
In particular, $D\bar F(v)$ is constant in radial directions, so
\begin{equation}\label{eq:D2Fradial}
D^{2}\bar F\big|_{q}(q,v)=0
\end{equation}
for any $v$.
Finally, the second derivatives are homogeneous of degree $-1$:
\begin{align}\label{eq:D2Fhom-1}
D^{2}\bar F\big|_{\lambda q}(v,v)&=\frac{d^{2}}{ ds^{2}}\bar F(\lambda q+sv)\big|_{s=0}\notag\\
&=\lambda\frac{d^{2}}{ ds^{2}}\bar F\left(q+\frac{s}{\lambda}v\right)\big|_{s=0}\notag\\
&=\lambda D^{2}\bar F\big|_{q}\left(\frac{v}{\lambda},\frac{v}{\lambda}\right)\notag\\
&=\frac{1}{\lambda}D^{2}\bar F_{q}(v,v).
\end{align}

The homogeneity of $\bar F$ and the identity \eqref{eq:D2Fhom-1} imply that $\bar FD^{2}\bar F$ is homogeneous of degree zero, so if we write $\pi(p)= \frac{(p,-1)}{\sqrt{1+|p|^{2}}}$, then
\begin{align}\label{eq:coeffhom}
F(p)D^{2}F\big|_{p}(v,v)&= \bar F(p,-1)D^{2}\bar F\big|_{(p,-1)}\left((v,0),(v,0)\right)\notag\\
&=\bar F(\pi(p))D^{2}\bar F\big|_{\pi(p)}\left((v,0),(v,0)\right)
\end{align}
The identity \eqref{eq:D2Fradial} then allows us to project  the vector $(v,0)$ onto the tangent space to the sphere at $\pi(p)$, giving (by the inequality \eqref{eq:D2Fsphere})
\begin{align}\label{eq:coeffproj}
m(p)F(p)&D^{2}F\big|_{p}(v,v)\notag\\
&= \left(\bar m\bar FD^{2}\bar F\right)\big|_{\pi(p)}\left(w-\left(w\cdot\pi(p)\right) \pi(p),w-\left(w\cdot\pi(p)\right)\pi(p)\right)\notag\\
&\geq A_{1}A_{2}\left|w-\left(w\cdot \pi(p)\right) \pi(p)\right|^{2}\notag\\
&= A_{1}A_{2}\left(\left|v-\left(v\cdot \hat p\right) \hat p\right|^{2}+\frac{(v\cdot \hat p)^{2}}{ 1+|p|^{2}}\right)\notag\\
&\geq A_{1}A_{2}\frac{|v|^{2}}{ 1+|p|^{2}},
\end{align}
where $\hat p = \frac{p}{ |p|}$ and we denoted $w=(v,0)$.  By Equation \eqref{eq:AMCFcoeff} this gives the result with $A=A_1A_2$. \end{proof}

Lemma \ref{lem:amcf}, along with Corollary \ref{periodic higherdimensional corollary} and Theorem \ref{periodic higherdimensional theorem}, are enough to prove a gradient estimate for periodic anisotropic mean curvature flows:

\begin{theorem}
Let $u: {\mathbb R}^n\times [0,T)\to{\RR}$ be a $\Gamma$-periodic regular solution to the
graphical anisotropic mean curvature flow \eqref{eq:AMCF} with $\osc u\leq M$.
Then $u$ satisfies the gradient estimate
\begin{equation*} 
1+|Du(x,t)|^2\le \exp\left(\frac{2M^2}{At}\right)
\end{equation*}
for all $x\in\RR^n$ and $t>0$, where $A$ is as given in Lemma \ref{lem:amcf}.
\end{theorem}

\section{Estimates for boundary value problems}\label{sec:bvp}

In this section the methods introduced above for the periodic case are adapted to
boundary value problems.  
A simple example is the Neumann problem on a convex
domain:

\begin{theorem}\label{thm:Neumann}
Let $\Omega\subset{\mathbb R}^n$ be a smoothly bounded convex domain, and let
$u$ be a regular solution of the Neumann
problem
\begin{align}\label{Neumann problem}
\begin{split}
\frac{\partial u}{\partial t}&=
a^{ij}(Du,t)D_iD_ju\\
D_{\mathbf n} u(x,t)&=0\quad\text{\rm for $x\in\partial\Omega$, $t>0$}\end{split}
\end{align}
where ${\mathbf n}$ is the outward-pointing unit normal vector to $\partial\Omega$ at $x$.
Let $\psi$ be a non-negative, concave modulus of continuity for $u(.,0)$, and
suppose that $\varphi$ is as in equation \eqref{conditions and equation for varphi}, with $\alpha$ defined by Equation \eqref{defn of alpha}.  Then $\varphi(.,t)$ is a modulus of continuity for $u(.,t)$ for each $t\geq 0$.\end{theorem}

\begin{proof}
Define as before 
$$
Z(x,y,t)=u(y,t)-u(x,t)-2\varphi\left(\frac{|y-x|}{ 2},{t}\right)-\varepsilon(1+t),
$$
and note that
$Z< 0$ when $t=0$, and when $|x-y|$ is small for any $t>0$.
If the first zero of $Z$ occurs with both $x$ and $y$ in the interior of $\Omega$, then the argument proceeds exactly as in the proof of Theorem \ref{periodic higherdimensional theorem}.  
Otherwise, we can assume that the first zero of $Z$ occurs at $(x_0,y_0,t_0)$ with $y_0\in\partial\Omega$, and let ${\mathbf n}(y)$ be the
outward unit normal to $\partial\Omega$ at $y$.  Then 
\begin{align}\label{eq:neumannderiv}
\frac{d}{ ds}Z(x,y+s{\mathbf n}(y),t)&=\nu\cdot
Du(y,t)-\varphi'\frac{y-x}{|y-x|}\cdot{\mathbf n}(y)\notag\cr
&=0-\varphi'\frac{y-x}{|y-x|}\cdot{\mathbf n}(y)\leq 0
\end{align}
since $D_{\mathbf n} u(y,t)=0$ by the boundary condition, and $\Omega$ is convex.
This leaves two possibilities:  If the inequality is strict, then $(x,y)$
is not a maximum point,  since $Z(x,y-s\nu,t)>Z(x,y,t)$ for $s$ small.
Otherwise, either $x$ is an interior point of $\Omega$, or
$x\in\partial\Omega$.  In the latter case we let $\nu'$ be the outward
unit normal to $\partial\Omega$ at $x$, and deduce that $D_{\nu'} Z\leq
0$, and therefore $D_{\nu'} Z=0$ since $(x,y)$ is a maximum point, and
$DZ=0$ at $(x,y)$.  Similarly, in the former case we have $DZ=0$.
It follows that $D^2Z\leq 0$, and we can argue exactly as in the case
where $x$ and $y$ are interior points to deduce a contradiction.\end{proof}

\begin{remark}
The barrier $\varphi$ need only be defined on $[0,L]\times[0,\infty)$, where $L\geq \frac12\diam\Omega$.  Gradient bounds under the assumption \eqref{degeneracy for alpha} follow exactly as in Corollary \ref{periodic higherdimensional corollary}.
\end{remark}

An estimate for the Dirichlet problem can be
obtained by the introduction of barriers to give an estimate at the
boundary.  We do this first in the case of a convex domain, where the method
works quite generally:

\begin{theorem}\label{thm:Dirichletconvex}
Let $\Omega$ be a bounded convex domain with smooth boundary, and let
$u$ be a regular solution of the Dirichlet problem
\begin{equation}\label{eq:Dirichlet}
\begin{split}
\frac{\partial u}{\partial t} &= a^{ij}(Du,t)D_{i}D_{j}u\\
u(x,t)&=0\quad\text{\rm for $x\in\partial\Omega$, $t>0$.}
\end{split}
\end{equation}
Let $\psi$ be a non-negative, concave modulus of continuity for $u(.,0)$, and
suppose that $\varphi$ is as in equation \eqref{conditions and equation for varphi}, with $\alpha$ defined by Equation \eqref{defn of alpha}.  Then $\varphi(.,t)$ is a modulus of continuity for $u(.,t)$ for each $t\geq 0$.
\end{theorem}

\begin{proof}
We first obtain an estimate near the boundary:  We will prove that
\begin{equation}\label{eq:bdryest}
|u(x,t)| \leq \varphi\left({d(x)},{t}\right)
\end{equation}
for all $x\in\Omega$ and $t>0$, where $d(x)=d(x,\partial\Omega)$.  Let $y\in\partial\Omega$, and an inward-pointing unit normal $\mathbf n$ to $\partial\Omega$ at $y$.
Define $v(x,t) = \varphi\left({(x-y)\cdot{\mathbf n}}, {t}\right)$.  Then (observing that $\varphi(.,t)$ is concave for each $t>0$) we have
$$
\frac{\partial v}{\partial t}-a^{ij}(Dv,t)D_{i}D_{j}v
= \left(\alpha(\varphi',t)-a^{ij}(\varphi'{\mathbf n},t){\mathbf n}_{i}{\mathbf n}_{j}\right)\varphi''\geq 0,
$$
since the definition of $\alpha$ in Equation \eqref{defn of alpha} with the choice $v=p$ gives
$$
\alpha(R,t)\leq a^{ij}(Rp,t)p_{i}p_{j}
$$
for any unit vector $p$.
Therefore the comparison principle yields $u(x,t)\leq v(x,t)$, since this is true initially and on the boundary.  Minimizing over $y\in\partial\Omega$ gives the estimate \eqref{eq:bdryest}.

Now let $Z(y,x,t)=u(y,t)-u(x,t)-2\varphi\left(\frac{|y-x|}{ 2},t\right)-\varepsilon(1+t)$.  As usual $Z<0$ for $t$ small and near the diagonal for all times.  We show that $Z<0$.  Otherwise there is a first time $t_0$ and points $x_0, y_0\in\Omega$ with $Z(x_0,y_0,t_0)=0$.

The key observation is that neither $x_0$ nor $y_0$ is in $\partial\Omega$:  In this case (say, $x_0\in\partial\Omega$) the estimate \eqref{eq:bdryest} gives
\begin{align*}
Z(y_0,x_0,t) &= u(y_0,t)-2\varphi\left(\frac{|y_0-x_0|}{ 2},t_0\right)-\varepsilon\\
&\leq \varphi\left({d(y_0)},t_0\right)-2\varphi\left(\frac{|y_0-x_0|}{ 2},t_0\right)-\varepsilon\\
&\leq \varphi\left(|y_0-x_0|,t_0\right)-2\varphi\left(\frac{|y_0-x_0|}{ 2},t_0\right)-\varepsilon\\
&< 0,
\end{align*}
since $\varphi(.,t_0)$ is increasing and concave.  Therefore $x_0$ and $y_0$ are in the interior of $\Omega$, and the argument is exactly as in Theorem \ref{periodic higherdimensional theorem}.
\end{proof}

\begin{remark}
It suffices for $\varphi$ to be defined on $[0,\diam\Omega]\times[0,T]$.  As before, gradient estimates follow under the assumption of Corollary \ref{periodic higherdimensional corollary}.
\end{remark}

In the case of mean curvature flow and other isotropic flows, the convexity of the boundary can be weakened to allow mean-convexity:

\begin{theorem}\label{Theorem 4}
Let $\Omega\subset{\mathbb R}^n$ be a smoothly bounded 
mean convex domain.  Let $u$ be a regular solution of the
Dirichlet problem
\begin{equation}\label{dirichlet problem}\begin{split}
\frac{\partial u}{\partial t}&=\left(
\alpha(|Du|,t)\frac{D_{i}uD_{j}u}{|Du|^{2}}+\beta(|Du|,t)\left(\delta_{ij}-\frac{D_iu D_ju}{ |Du|^2}\right)
\right)D_iD_ju\\
u(x,t)&=0\quad\text{\rm for $x\in\partial\Omega$, $t>0$.}\end{split}\end{equation}
Let $\psi$ be a non-negative, concave modulus of continuity for $u(.,0)$, and
suppose that $\varphi$ is as in equation \eqref{eq:1Dreduction} with $\varphi(z,0)\geq \psi(z)$ for $z>0$.  Then $\varphi(.,t)$ is a modulus of continuity for $u(.,t)$ for each $t\geq 0$.
\end{theorem}

\begin{proof}
As in Theorem \ref{thm:Dirichletconvex} the key step is to obtain an estimate near the boundary.  We prove that 
$$
|u(x,t)|\leq \varphi\left(d(x),{t}\right)
$$
for all $x\in\Omega$ and $t>0$, where $d(x)=d(x,\partial\Omega)=\inf\{d(x,y):\ y\in\partial\Omega\}$.  

We recall some properties of the function $d$:   For $y\in\partial\Omega$, let ${\mathbf n}(y)\in S^{n-1}$ be the inward-pointing unit normal to $\partial\Omega$, and consider the map $f:\partial\Omega\times\RR\to\RR^{n}$ given by $f(y,s)=y+s{\mathbf n}(y)$.  Then for every $y\in\partial\Omega$ there exists a $\rho(y)>0$ such that
$d(f(y,s))=s$ for $0\leq s\leq \rho(y)$, but $d(f(y,s))<s$ for $s>\rho(y)$.  The map $f$ takes the set
$\left\{(y,s):\ y\in\partial\Omega,\ 0<s\leq\rho(y)\right\}$ onto $\Omega$.  Furthermore $f$ is injective and has injective derivative on the set ${\mathcal O}=\left\{(y,s):\ y\in\partial\Omega,\ 0<s<\rho(y)\right\}$.  

The function $d$ is smooth on $f({\mathcal O})$ and non-smooth elsewhere.    A direct computation (see for example \cite[Lemma 14.17]{trudinger}) shows that on $f({\mathcal O})$, 
\begin{equation}\label{eq:D1d}
Dd(f(y,s)) = n(y)
\end{equation}
and
\begin{equation}\label{eq:D2d}
D^{2}d(w,w) = -\langle w, A_{y}\circ (I-sA_{y})^{-1}(w)\rangle
\end{equation}
for any $w\in T_{y}\partial\Omega$, where $A_{y}(w)=-D{\mathbf n}(w)\in T_{y}\partial\Omega$ for any $w\in T_{y}\partial\Omega$.

There are two circumstances under which the function $d$ can fail to be smooth at a point $x\in\Omega$:  Either there exist two points $y_{1}$ and $y_{2}$ in $\partial\Omega$ such that
$x=f(y_{1},d(x))=f(y_{2},d(x))$, or there is a single point $y$ with $x=f(y,d(x))$, but for some non-zero $v\in T_{y}\partial\Omega$, $Df\big|_{(y,d(x))}(v,0)=0$.  Importantly for our purposes, in neither of these circumstances can there exist a $C^{2}$ function $g$ on $\Omega$ such that $g\leq d$ on $\Omega$ but $g(x)=d(x)$:  If this were the case we could write
\begin{equation}\label{eq:localg}
g(z) = g(x) + {\mathbf a}\cdot (z-x) + \frac12 M(z-x,z-x) + o(y-x)^{2}
\end{equation}
as $|z-x|\to 0$, for some ${\mathbf a}\in\RR^{n}$ and some symmetric bilinear form $M$.
Note that $|a|\leq 1$, since by the triangle inequality we have
$$
g(x+sa)\leq d(x+s{\mathbf a}) \leq d(x)+d(x,x+s{\mathbf a}) = g(x)+s|{\mathbf a}|,
$$
and applying Equation \eqref{eq:localg} on the left-hand side gives
$$
g(x)+s|{\mathbf a}|^{2} \leq g(x)+s|{\mathbf a}|+o(s)
$$
as $s\to 0$.  
It follows that whenever $x=f(y,s)$ with $s\leq\rho(y)$, we necessarily have ${\mathbf a}={\mathbf n}(y)$:  Taking $z=x-s{\mathbf n}(y)$ in Equation \eqref{eq:localg} gives
\begin{align*}
d(x) - s{\mathbf a}\cdot {\mathbf n}(y) +o(s) 
&= g(x)+s{\mathbf a}\cdot {\mathbf n}(y)+o(s)\\
&= g(x-s{\mathbf n}(y))+o(s)\\
&\leq d(x-s{\mathbf n}(y))+o(s)\\
& = d(x)-s
\end{align*}
as $s\to 0$, and hence ${\mathbf a}\cdot {\mathbf n}(y)=1$ and ${\mathbf a}={\mathbf n}(y)$.
Thus if there are two points $y_{1}$ such that $f(y_{1},d(x))=f(y_{2},d(x))=x$, we would have
$$
{\mathbf n}(y_{1})={\mathbf a} = {\mathbf n}(y_{2}).
$$
But then $y_{1}= x-d(x){\mathbf n}(y_{1})=x-d(x){\mathbf n}(y_{2}) = y_{2}$, ruling out the first case of non-smoothness of $d$.

In the second case, $Df$ is not injective, so there exists some unit vector $w\in T_{y}\partial\Omega$ such that 
$0=Df(w,0) = w+A_{y}(w)$.  The triangle inequality gives for any $\gamma>0$
\begin{align*}
d(x+\delta w)
&\leq d(x-\gamma{\mathbf n}(y)+\delta w) + d(x-\gamma{\mathbf n}(y)+\delta w,x+\delta w)\\
&=d(x-\gamma{\mathbf n}(y)+\delta w)+\gamma.
\end{align*}
For each $\gamma\in(0,d(x))$, $d(x)-\gamma<\rho(y)$, so $d$ is smooth at $x-\gamma{\mathbf n}(y)$, and
\begin{align*}
d(x-\gamma{\mathbf n}(y)+\delta w) 
&= d(x-\gamma{\mathbf n}(y))+\frac{\delta^{2}}2 D^{2}d(w,w)+o(\delta^{2})\\
&=d(x)-\gamma+\frac{\delta^{2}}2 D^{2}d(w,w)+o(\delta^{2})
\end{align*}
as $\delta\to 0$.  Combining these two inequalities and noting that Equation \eqref{eq:D2d} gives
$D^{2}d\big|_{x-\gamma{\mathbf n}(y)}(w,w) = -\frac{1}{\gamma}$, we have
$$
d(x+\delta w)\leq d(x)-\frac{\delta^{2}}{ 2\gamma}+o(\delta^{2}).
$$
Taking $z=x+\delta w$ in Equation \eqref{eq:localg} gives
\begin{align*}
d(x)-\frac{\delta^{2}}{ 2\gamma}&=d(x+\delta w)+o(\delta^{2})\\
&\geq g(x+\delta w)+o(\delta^{2})\\
&= g(x) +\frac{\delta^{2}}{ 2}M(w,w)+o(\delta^{2})
\end{align*}
as $\delta\to 0$, and hence $M(w,w)\leq -\frac{1}{\gamma}$.  Taking $\gamma\to 0$ gives a contradiction.

Now we proceed to the estimate:  Define
$$
Z_{B}(x,t) = u(x,t)-\varphi\left(d(x),t\right)-\varepsilon(1+t).
$$
On the boundary of $\Omega$, $Z_{B}(x,t)<0$.  Also, by regularity of $u$ and the
initial condition for $\varphi$, $Z_{B}(x,t)<0$ everywhere on $\Omega$ for small $t$. 
Consider the first time $t_0>0$ and $x_0\in\Omega$ such that $Z_B(x_0,t_0)=0$, if one exists.

Note that $z\mapsto \varphi(z,t_0)$ is strictly increasing and $C^2$, and so has a $C^2$ inverse which we denote $\mu$.  Then $g(x) = \mu\left(u(x,t_0)-\varepsilon(1+t_0)\right)$ is $C^2$, and we have
$g\leq d$ on $\Omega$, but $g(x_0)=d(x_0)$.  The observations above imply that $d$ is smooth at $x_0$, and hence so is $Z_{B}$.  The spatial derivatives for $Z_B$ are:
\begin{gather*}
\pd {Z_B} {y^i}=\pd u{y^i} -\varphi' D_i d  \\
\pdd {Z_B}{y^i}{y^j}=\pdd u {y^i}{y^j}  - \varphi''D_idD_jd-\varphi' D_iD_j d.
\end{gather*}
Thus $Du(x_0) = \varphi' Dd(x_0)$, and 
$D_{i}D_{j}u(x_0) \leq \varphi''D_{i}d D_{j}d + \varphi' D_{i}D_{j}d$.  Also
\begin{align*}
0\leq \frac{\partial}{\partial t}Z_{B}(x_0,t_0) 
&= \frac{\partial u}{\partial t}-\frac{\partial\varphi}{\partial t}-\varepsilon\\
&<\left(\alpha(|Du|,t)\frac{u_{i}u_{j}}{ |Du|^{2}}+\beta(|Du|,t)\!\left(\delta_{ij}-\frac{u_{i}u_{j}}{ |Du|^{2}}\right)\right)D_{i}D_{j}u-\frac{\partial\varphi}{\partial t}\\
&\leq\alpha(\varphi',t)D_{i}dD_{j}d\left(\varphi''D_{i}dD_{j}d+\varphi'D_{i}D_{j}d\right)\\
&\quad\null+\beta(\varphi',t)\!\left(\delta_{ij}-D_{i}dD_{j}d\right)\!
\left(\varphi''D_{i}dD_{j}d+\varphi'D_{i}D_{j}d\right)
\!-\frac{\partial\varphi}{\partial t}\\
&=\alpha(\varphi',t)\varphi''+\beta(\varphi',t)\Delta d-\frac{\partial\varphi}{\partial t}\\
&\leq 0
\end{align*}
since $|Dd|=1$, $D_idD_{i}D_jd=0$, and $\Delta d = \sum_{i}D_iD_id\le0$.  The last is because (as in Lemma 14.17 of \cite{trudinger})
$\Delta d(x)=\sum_{i=1}^{n-1} \frac{-\kappa_i(y)}{1-\kappa_i(y) d(x)}$ where $\kappa_{1}(y),\dots,\kappa_{n-1}(y)$ are the principal curvatures of $\partial\Omega$ at $y$.  Since $d$ is smooth at $x_0$, $\kappa_i(y) d(x)<1$, and
\begin{equation*} 
  \sum_{i=1}^{n-1} \frac{-\kappa_i}{1-\kappa_i d}\,\le\, -\negthickspace\negthinspace\sum_{i=1}^{n-1} \kappa_i\,\le \,0,
\end{equation*}
by the mean-convexity of $\partial\Omega$.   This proves that $Z_B(x,t)\le 0$ for all $x$ and $t$ and all $\varepsilon>0$.  The remainder of the proof is exactly as for Theorem \ref{thm:Dirichletconvex}.\end{proof}

We conclude this section with a gradient estimate for solutions of the Dirichlet problem for the anisotropic mean curvature flow.  We consider solutions of \eqref{eq:AMCF} with coefficients given by  \eqref{eq:AMCFcoeff}, and assume that the functions $\bar F$ and $\bar m$ satisfy bounds as follows:  There exists $C>0$ such that for each $z$ in the unit sphere $S^{n}\subset \RR^{n+1}$, and all vectors $u$, $v$ and $w$ in $T_{z}S^{n}$,  
\begin{align}
\frac{1}{ C}\leq \bar F(z)&\leq C;\label{eq:Fbound}\\
\frac{1}{ C}\leq \bar m(z)&\leq C;\label{eq:mbound}\\
\frac{|v|^{2}}{ C}\leq D^{2}\bar F\big|_{z}(v,v)&\leq C|v|^{2};\label{eq:D2Fbound}\\
|D^{3}\bar F\big|_{z}(u,v,w)|&\leq C|u||v||w|.\label{eq:D3Fbound}
\end{align}
Note that conditions \eqref{eq:Fbound}--\eqref{eq:D3Fbound} are consequences of the assumptions on $F$ and $m$ made in the introduction (after Equation \eqref{eq:AMCFcoeff}).
We also assume that $\partial\Omega$ satisfies a lower curvature bound, so that the anisotropic
principal curvatures (defined below) satisfy $\kappa_{i}\geq -C_{1}$. 

\begin{theorem}\label{thm:AMCFDirichlet}
Let $u: \bar\Omega\times[0,T)\to[-M/2,M/2]$ be a regular solution of the anisotropic mean curvature flow
\eqref{eq:AMCF}--\eqref{eq:AMCFcoeff} with Dirichlet boundary condition $u(x,t)=0$ for $x\in\partial\Omega$, $t\geq 0$.  If $\partial\Omega$ is smooth and the anisotropic mean curvature of $\partial\Omega$ is non-negative, then 
$$
1+|Du(x,t)|^2\leq C\exp\left(\frac{M^2}{At}\right)
$$
for all $x\in\bar\Omega$ and $t>0$, where $C$ depends on $F$ (up to third derivatives) and the lower curvature bound $C_1$, and $A$ is as in Lemma \ref{lem:amcf}.
\end{theorem}

The precise definition of anisotropic mean curvature is given below.

\begin{proof}
We begin with construction of a suitable function to replace the distance to the boundary as used in Theorem \ref{Theorem 4}.   This is done using the distance defined by a Finsler metric as follows:  Let $\tilde F$ be the restriction of $\bar F$ to $(\RR^{n})^{*}$, so that $\tilde F(p)=\bar F(p,0)$.  Then $\tilde F$ is homogeneous of degree one, smooth and positive away from the origin, and strictly convex in non-radial directions.  There is a natural dual function defined on $\RR^{n}$ by
$$
F^{*}(v)=\sup\{v(p):\ \tilde F(p)\leq 1\}.
$$
Then $F^{*}$ is also homogeneous of degree one, smooth and positive away from the origin, and strictly convex in non-radial directions.  In particular this defines a (possibly non-symmetric) distance function, and a natural generalisation of the distance to the boundary, given by 
$$
d(x) =\inf\{ F^{*}(x-y):\ y\in\partial\Omega\}.
$$
The duality between $\tilde F$ and $F^{*}$ produces natural maps between $\{F^{*}=1\}\subset \RR^{n}$ and $\{\tilde F=1\}\subset(\RR^{n})^{*}$:  For ${\mathbf n}\in\RR^{n}\backslash\{0\}$, define
$$
p({\mathbf n}) = DF^{*}\big|_{{\mathbf n}}.
$$
This map is homogeneous of degree zero, and maps $\{F^{*}=1\}$ diffeomorphically to 
$\{\tilde F=1\}$.  The inverse map is given by
$$
{\mathbf n}(p) = D\tilde F\big|_{p}
$$
for $p\in(\RR^{n})^{*}\backslash\{0\}$.

The properties of $d$ are very closely analogous to those of the usual distance to the boundary.  There is a smooth map ${\mathbf n}$ from $\partial\Omega$ to $\{F^{*}=1\}$ given by the requirement that $p({\mathbf n}(y))(w)=0$ for $w\in T_{y}\partial\Omega$, and that ${\mathbf n}(y)$ points into $\Omega$ at $y$.  From this we construct a smooth map
$f: \partial\Omega\times\RR\to\RR^{n}$ by $f(y,s)=y+s{\mathbf n}(y)$.  As in the isotropic case, for each $y\in\partial\Omega$ there exists $\rho(y)>0$ such that $d(f(y,s))=s$ for $0\leq s\leq\rho(s)$, and the map $f$ is injective and has injective derivative on ${\mathcal O}=\{(y,s):\ 0<s<\rho(y)\}$, and surjective (onto $\Omega$) on $\bar{\mathcal O}=\{(y,s):\ 0<s\leq\rho(s)\}$.  The derivative of ${\mathbf n}$ defines the anisotropic curvature $A_{y}$ at $y$ by the formula 
$$
D{\mathbf n}\big|_{y}(w)=A_{y}(w).
$$
The anisotropic curvature $A_{y}$ is a linear map from $T_{y}\partial\Omega$ to itself, which is symmetric with respect to the anisotropic metric tensor $g=D^{2}F^{*}\big|_{{\mathbf n}(y)}$.  The eigenvalues $\kappa_{1},\dots,\kappa_{n-1}$ of $A_{y}$ are the anisotropic principal curvatures, and their sum (the trace of $A_{y}$) is the anisotropic mean curvature of $\partial\Omega$.

The derivative of $f$ in directions tangent to $\partial\Omega$ is given by
$$
Df\big|_{y,s}(w,0)=(I+sA_{y})(w),
$$
and therefore in $f({\mathcal O})$ we have $Dd\big|_{f(y,s)}({\mathbf n}(y))=1$ and
$Dd\big|_{f(y,s)}(w)=0$ for $w\in T_{y}\partial\Omega$, while
\begin{equation}\label{eq:D2dFn}
D^{2}d\big|_{f(y,s)}({\mathbf n}(y),u)=0
\end{equation}
for any $u$, and
\begin{equation}\label{eq:D2dFt}
D^{2}d\big|_{f(y,s)}(w,w)=-g(A_{y}(w),(I+sA_{y})^{-1}(w))
\end{equation}
for all $w\in T_{y}\partial\Omega$.
The anisotropic mean curvature of a level set of $d$ is given by the trace with respect to $g$, which is 
$$
\sum_{i=1}^{n-1}\frac{\kappa_{i}}{1-s\kappa_{i}},
$$
exactly as in the isotropic case.  In particular, if $\partial\Omega$ has non-negative anisotropic mean curvature, then so does the level set of $d$ through each point of $f({\mathcal O})$.  Note also that
the same argument as in the isotropic case shows that if there exists a smooth function $g$ on $\Omega$ with $g\leq d$ on $\Omega$ but $g(x)=d(x)$ for some $x\in\Omega$, then $d$ is smooth at $\Omega$, there exists a unique $y\in\partial\Omega$ with $x=f(y,d(x))$, and $d(x)<\rho(x)\leq \frac{1}{\max_{i}\kappa_{i}(y)}$.

Now we attempt to proceed exactly as in Theorem \ref{Theorem 4}.  There is one further complication to overcome, which arises purely in obtaining the boundary estimate.  Define $Z_{B}:\bar\Omega\times(0,T)\to\RR$ by
$$
Z_{B}(x,t)=u(x,t)-2\varphi\left(\frac{d(x)}{2},t\right)-\varepsilon(1+t).
$$
We will choose $\varphi$ to be a supersolution of a suitable one-dimensional parabolic equation, with $\varphi(z,t)\geq \frac{M}{ 2}$ as $t\to 0$ for any $z>0$, and $\varphi(0,t)=0$.  In particular,  at each time $t>0$, $\varphi(.,t)$ will be strictly increasing and concave.  If the maximum of $Z_{B}$ at time $t$ is  non-negative, then it is attained at a point $x$ in $\Omega$ at which $d$ is smooth, and there is a unique $y\in\partial\Omega$ such that $f(y,d(x))=x$.  Then the first and second spatial derivatives are given by
$$
DZ_{B}= Du(x,t)-\varphi' Dd(x);
$$
so that $Du(x,t) =\varphi' p$, where $p=DF^{*}\big|_{{\mathbf n}(y)}\in\{\tilde F=1\}$, and
$$
D_{i}D_{j}Z_{B}=D_{i}D_{j}u(x,t)-\frac12\varphi'' D_{i}d D_{j}d - \varphi' D_{i}D_{j}d,
$$
so that $D_iD_ju(x,t)\leq \varphi'' p_{i}p_{j}
+\varphi' D_{i}D_{j}d$.  The time derivative of $Z_{B}$ at the maximum point is given by
\begin{align}
\frac{\partial Z_{B}}{\partial t}(x,t)&=a^{ij}(\varphi'p)D_{i}D_{j}u(x,t)
-2\frac{\partial\varphi}{\partial t}-\varepsilon\notag\\
&< \frac12a^{ij}(\varphi'p)p_{i}p_{j}\varphi''
+\varphi'a^{ij}(\varphi'p)D_{i}D_{j}d(x)-2\frac{\partial\varphi}{\partial t}\label{eq:amcfdtbd}
\end{align}
Lemma \ref{lem:amcf} gives $a^{ij}(\varphi'p)p_{i}p_{j}\geq \frac{A}{ 1+(\varphi')^{2}}$, so (since $\varphi''\leq 0$) we have
$$
a^{ij}(\varphi'p)p_{i}p_{j}\varphi''\leq \frac{A\varphi''}{ 1+(\varphi')^{2}}.
$$
To control the second term in \eqref{eq:amcfdtbd}, we note that $D^{2}d({\mathbf n}(y),.)=0$, and choose coordinates such that
$e_{1},\dots,e_{n-1}$ span $T_{y}\partial \Omega = T_{{\mathbf n}(y)}\{F^{*}=1\}$ and are eigenvectors of $A_{y}$, and $e_{n}={\mathbf n}(y)=D\tilde F\big|_{p}$, orthonormal with respect to $g=D^{2}F^{*}\big|_{{\mathbf  n}(y)}$.  The dual basis $\{\phi^{1},\dots,\phi^{n}\}$ is then as follows:  
$\phi^{n}=p({\mathbf n}(y))=DF^{*}\big|_{{\mathbf n}(y)}=p$, while $\phi^{1},\dots,\phi^{n-1}$ are a basis for $T_{p}\{\tilde F=1\}$.  
In these coordinates we have for $i=1,\dots,n$
$$
D_{n}D_{i}d = D^{2}d({\mathbf n}(y),e_{i})=0
$$
by Equation \eqref{eq:D2dFn}.  Equation \eqref{eq:D2dFt} gives the following precise expression for  the remaining components of $D^{2}d$:
$$
D_{i}D_{i}d = -\frac{\kappa_{i}}{ 1-d\kappa_{i}}
$$
for $1\leq i\leq n-1$, and $D_{i}D_{j}d=0$ for $1\leq i<j\leq n-1$.  The term then becomes
$$
\varphi'a^{ij}(\varphi'p)D_{i}D_{j}d
= -\varphi'\bar m(\varphi' p,-1)\sum_{i=1}^{n-1}(\bar FD^{2}\bar F)\big|_{(p,-1/\varphi')}(\phi^{i},\phi^{i})\frac{\kappa_{i}}{ 1-d\kappa_{i}}.
$$

The orthonormality of the basis implies that 
$$
\delta^{ij}=D^{2}\tilde F\big|_{p}(\phi^{i},\phi^{j})=D^{2}\bar F\big|_{(p,0)}(\phi^{i},\phi^{i}).
$$
The bound of Equation \eqref{eq:D3Fbound} implies that
$$
\left|(\bar FD^{2}\bar F)\big|_{(p,-1/\varphi')}(\phi^{i},\phi^{i})-1\right|\leq\frac{C}{\varphi'}
$$
for each $i$.  This implies
\begin{align*}
\varphi'a^{ij}(\varphi'p)D_{i}D_{j}d
&\leq m\left(-\left(\varphi'-C\right)\sum_{\kappa_{i}>0}\frac{\kappa_{i}}{ 1-d\kappa_{i}}
-\left(\varphi'+C\right)\sum_{\kappa_{i}<0}\frac{\kappa_{i}}{ 1-d\kappa_{i}}\right)\\
&\leq m\left(-\left(\varphi'-C\right)\sum_{i=1}^{n-1}\frac{\kappa_{i}}{ 1-d\kappa_{i}}
-2C\sum_{\kappa_{i}<0}\frac{\kappa_{i}}{ 1-d\kappa_{i}}\right)\\
&\leq 2C^{2}C_{1}(n-2),
\end{align*}
provided $\varphi'\geq C$.  From this we see that 
$$
\frac{\partial Z_{B}}{\partial t}\leq \frac{A}{2(1+(\varphi')^{2})}\varphi''+2(n-2)C^{2}C_{1}-2\varphi_{t},
$$
and a gradient bound follows provided we can find $\varphi$ such that the right-hand side is non-positive:  Precisely, we need
\begin{align*}
\varphi_{t}&\geq \frac{A}{4(1+(\varphi')^{2})}\varphi'' + (n-2)C^{2}C_{1};\\
\varphi(t,0)&=0,\quad t>0;\\
\varphi(z,t)&\geq \frac{M}{4},\quad t\to0,\ z>0;\\
\varphi'&\geq C;\\
\varphi(.,t)&\ \text{\rm increasing and concave, $t>0$.}
\end{align*}
It suffices to choose $\varphi$ to be a supersolution for the equation
\begin{equation}\label{eq:super1}
\varphi_{t}=\frac{A\varphi''}{ 1+(\varphi')^{2}}+B
\end{equation}
with $\varphi'\geq C$, where $C$ depend on the equation only, and $B$ depends on the equation and on the lower curvature bound $C_1$, and $A$ is as in Lemma \ref{lem:amcf}.

This establishes the estimate near the boundary, and the remainder of the proof proceeds exactly as in Theorem \ref{periodic higherdimensional theorem}, by considering the function
$$
Z(y,x,t)=u(y,t)-u(x,t)-2\varphi\left(\frac{|y-x|}{2},t\right)-\varepsilon(1+t),
$$
where $\varphi$ satisfies
\begin{equation}\label{eq:super2}
\frac{\partial\varphi}{\partial t}\geq \frac{A\varphi''}{1+(\varphi')^2}.
\end{equation}
A function $\varphi$ which is a supersolution for both \eqref{eq:super1} and \eqref{eq:super2} can be constructed using translating solutions as in Theorem 7 of \cite{ac:gradest1D}, yielding the explicit gradient estimate.
\end{proof}

The use of more sophisticated barriers at the boundary produces estimates for more general
Dirichlet problems.  In particular, for the (anisotropic) mean curvature flow on an (anisotropic) mean-convex domain one can allow arbitrary prescribed $C^{1,1}$ boundary
data, and gradient estimates will follow depending on the $C^{1,1}$
norm of the boundary data as well as the oscillation of the solution.
However, examples show that there
are no gradient estimates which hold up to the boundary if the boundary data is only Lipschitz.  Similarly, the condition of mean-convexity cannot be relaxed. 
    
\section{ Applications}\label{sec:app}

In this section we give some applications of the estimates derived above
to prove existence for solutions of initial and boundary value problems with highly singular
initial data.  

In the case of the graphical mean curvature flow, the interior estimates of Ecker and Huisken \cite{eh:interior} yielded the existence of entire solutions for any locally Lipschitz initial data.  Angenent \cite{ang:mcf} extended this argument using the estimates of Evans and Spruck \cite{es:gradest} to yield existence of entire solutions for any continuous initial data.

We will provide a somewhat weaker result for the entire problem for anisotropic mean curvature flows, as well as results for periodic initial data, and homogeneous Neumann and Dirichlet problems.

\subsection{Periodic initial value problems}

We consider equations of the form \eqref{evolution equation}, where $a^{ij}$ is positive definite and locally Lipschitz continuous, and $b$ is locally H\"older continuous on $(\RR^{n})^{*}\times[0,\infty)$, and assume that there exists a non-negative function $\varphi$ which is continuous on $[0,\infty)\times[0,\infty)\setminus\{(0,0)\}$ and nondecreasing and concave in the first argument, which is regular on $[0,\infty)\times(0,\infty)$ and satisfies Equation \eqref{conditions and equation for varphi} with $\alpha$ defined by
\eqref{defn of alpha} and initial condition $\psi$. 

\begin{theorem}\label{thm:existperiodic}
Let $u_{0}$ be a $\Gamma$-periodic continuous function on $\RR^{n}$ with modulus of continuity $\psi$.  Then there exists a unique $\Gamma$-periodic $u$ which is continuous on $\RR^n\times[0,\infty)$, regular in $\RR^{n}\times(0,\infty))$, satisfies Equation \eqref{evolution equation} for $t>0$, and has
$u(x,0)=u_{0}(x)$ for all $x\in\RR^{n}$.
\end{theorem}

In particular, $\varphi$ exists for any bounded $\psi$ if $a^{ij}(p,t)v_{i}v_{j}\geq C|p|^{-2}|v|^{2}$ for $|p|$ large, or for $\psi(z)\leq Kz^\beta$ if $a^{ij}(p,t)v_iv_j\geq C|p|^{-\gamma}|v|^2$ with $\gamma<\frac{2}{1-\beta}$ for $|p|$ large (compare section 6 of \cite{ac:gradest1D}).

\begin{proof}
We will produce a solution as a limit of a sequence of solutions with smooth initial data given by mollifications:
$$
u_{0,n}(x) = \int_{\RR^{n}}\rho(z) u_{0}(x+z/n)dz^{n},
$$ 
where $\rho$ is a smooth, non-negative, rotationally symmetric function with compact support in the unit ball and with 
$\int_{\RR^{n}}\rho(z) dz^{n}=1$.   Note that $u_{0,n}$ has modulus of continuity $\psi$ for every $n$.

For each $n$ there exists a regular solution $u_{n}$ of \eqref{evolution equation}, by a straightforward modification of the proof of Theorem 12.16 in \cite{lieb:book}.  By Theorem 
\ref{periodic higherdimensional theorem} we have $|u_n(z,t)|\leq C(\psi,t_0)$ any each set $\RR^n\times[t_0,\infty)$.  In particular, on each of these sets the equation is uniformly parabolic.  Theorem 12.3 of \cite{lieb:book} provides a H\"older estimate on the first spatial derivative on each of these sets, and then a Schauder estimate (such as in Theorem 4.9 of \cite{lieb:book} provides a $C^{2,\alpha}$ estimate.  All of these estimates are independent of $n$.

A subsequence converges in $C_{\tx{loc}}^{2,\beta}((\RR^{n})\times(0,\infty))$ to a limit $u$ which is a solution of the equation.  We prove that $u(.,t)\to u_0$ uniformly as $t\to 0$.

Since $u_0$ is uniformly continuous, there exists a non-increasing function $v$ such that
$|u_0(y)-u_0(x)|\leq \varepsilon+v(\varepsilon)|y-x|$ for all $x$ and $y$, for any $\varepsilon>0$.
The functions $u_{0,n}$ also satisfy
$|u_{0,n}(y)-u_{0,n}(x)|\leq \varepsilon+v(\varepsilon)|y-x|$ for all $\varepsilon>0$ and $x,y\in\RR^n$.
Define $\omega(z)=\inf\{\varepsilon+v(\varepsilon)z:\ \varepsilon>0\}$.  Then $\omega$ is a concave modulus of continuity for $u_0$ (hence also for $u_{0,n}$) with $\omega(z)\to 0$ as $z\to 0$.

The modulus of continuity $\omega$ also controls how far $u_{0,n}$ varies from $u_{0}$:
\begin{align}
|u_{0,n}(x) -u_{0}(x)| &= \left|\int_{\RR^{n}}\rho(z)\left(u_{0}(x+z/n)-u_{0}(x)\right)dz^{n}\right|\notag\\
&\leq \int_{\RR^{n}}\rho(z)\left|u_{0}(x+z/n)-u_{0}(x)\right|dz^{n}\notag\\
&\leq \int_{\RR^{n}}\rho(z)\omega(|z|/n)dz^{n}\notag\\
&\leq \omega(1/n).\label{eq:mollifnear}
\end{align}

By continuity of $a$ and $b$, given $K>0$ there exist $\Lambda_{K}$ and $\mu_{K}$ such that 
$|b(p,t)|\leq \mu_{K}$ and $a^{ij}(p,t)\xi_{i}\xi_{j}\leq \Lambda_{K}|\xi|^{2}$
 for all $\xi$,  if $|p|\leq K$ and $0\leq t\leq 1$.
 
 Let $\phi:\ [0,\infty)\to\RR$ be the unique solution of
 $$
 \phi''(z)+(n-1)\frac{\phi'(z)}{|z|} = \frac12\left(\phi(z) -z\phi'(z)\right)
 $$
 with $\phi'(0)=0$ and $\phi'(z)\to 1$ as $z\to\infty$.  Then 
 for each $a>0$ we can construct a supersolution of \eqref{evolution equation} by setting
 \begin{equation}\label{eq:supersoln}
w_{a}(x,t) = u_{0,n}(x_{0}) + a + \mu_{v(a)}t + v(a)\sqrt{\Lambda_{v(a)}t}\phi\left(
\frac{|x-x_{0}|}{\sqrt{\Lambda_{v(a)}t}}\right).
\end{equation}
Note that $w_{a}(x,t)\to u_{0,n}(x_{0})+a+v(a)|x-x_{0}|\geq u_{0,n}(x_{0})+\omega(|x-x_{0}|)\geq u_{0,n}(x)$ as $t\to 0$.  Therefore by comparison we have $w_{a}(x,t)\geq u_{n}(x,t)$ for each $t>0$ and $a>0$.  In particular,
$$
u_{n}(x_{0},t)-u_{0,n}(x_{0})\leq a + \mu_{v(a)}t + v(a)\sqrt{\Lambda_{v(a)}t}\phi(0).
$$
A similar subsolution gives a bound from below.
This establishes a uniform estimate on continuity in time for $u_{n}$:  Given $\varepsilon>0$, let 
$$
\delta(\varepsilon)=\max\left\{\frac{\varepsilon}{\mu_{v(\varepsilon)}},\frac{\varepsilon^{2}}{v(a)^{2}\Lambda_{v(a)}\phi(0)^{2}}\right\}.
$$  
Then for $0\leq t\leq \delta(\varepsilon)$ and all $x\in\RR^{n}$,
\begin{equation}\label{eq:timecont}
|u_{n}(x,t)-u_{0,n}(x)|\leq 3\varepsilon.
\end{equation}
Uniform convergence of $u(.,t)$ to $u_{0}$ follows, since for each $t>0$ and $n\in\NN$,
$$
|u(.,t)-u_{0}|_{\infty}\leq |u(.,t)-u_{n}(.,t)|_{\infty} + |u_{n}(.,t)-u_{n,0}|_{\infty}+|u_{n,0}-u_{0}|_{\infty}.
$$
The first term approaches zero as $n\to\infty$ along our subsequence, the second is bounded by \eqref{eq:timecont}, and the last approaches zero by \eqref{eq:mollifnear}, so that
$$
|u(x,t)-u_{0}(x)|\leq \varepsilon
$$
for $0\leq t\leq \delta(\varepsilon/3)$.  Uniqueness is trivial in this case.
\end{proof}

%%% Examples:  MCF or AMCF (any cts data); nonparam MCF with suff Holder reg

\subsection{The Neumann problem}

The result for the Neumann problem on a convex domain is almost identical to that for the periodic problem.  The assumptions on $a^{ij}$, $b$, $\psi$ and $\varphi$ are the same as for Theorem \ref{thm:existperiodic}.

\begin{theorem}\label{thm:existneumann}
Let $\Omega$ be a bounded convex domain with $C^{2,\alpha}$ boundary.
Let $u_{0}\in C^{0}(\bar\Omega)$.  Then there exists a unique $u$ which is a regular solution of \eqref{Neumann problem} in $\bar\Omega\times(0,\infty))$ and is continuous on $\bar\Omega\times[0,\infty)$ with $u(x,0)=u_{0}(x)$ for all $x\in\bar\Omega$.
\end{theorem}

\begin{proof}
The proof follows closely that for Theorem \ref{thm:existperiodic}.  

First construct a sequence of smooth initial data $u_{0,n}$ with $D_{{\mathbf n}}u_{0,n}=0$ on $\partial\Omega$, converging uniformly to $u_{0}$ as $n\to\infty$.   For example, such a sequence can be constructed as follows:  Choose an open cover of $\partial\Omega$ by open sets $U_{i}$ with diffeomorphisms $\psi_{i}:\ V_{i}\subset\RR^{n-1}\to U_{i}$.  Choose $r>0$ such that each point $x$ in $\Omega$ of distance less than $r$ from $\partial\Omega$ has a unique closest point $\pi(x)$ in $\partial\Omega$.  Let $W_{i} = \left\{x\in\bar\Omega:\ d(x,U_{i})=d(x,\partial\Omega)<r\right\}$.   For each $i$ there is a natural diffeomorphism $\tilde\psi_{i}$ from $Z_{i}=V_{i}\times[0,r)\to W_{i}$, given by $\tilde\psi_{i}(y,s) = \psi_{i}(y)+s{\mathbf n}(\psi_{i}(y))$, where ${\mathbf n}(x)$ is the inward-pointing unit normal to $\partial\Omega$ at $x$.  $\bar\Omega$ is covered by the collection of relatively open sets $\{\Omega,W_{1},\dots,\}$.  Choose a partition of unity subordinate to this cover, consisting of smooth functions $f_{0}$ with support in $\Omega$ and $f_{i}$ with support in each $W_{i}$.  Set $\delta = \min\{d(\tx{supp}f_{0},\partial\Omega),\ d(\tx{supp}f_{i}\circ\tilde\psi_{i},\partial Z_{i})\}$.  Then for $n>\delta^{-1}$ define 
$$
u_{0,n}(x)=f_{0}(x)\int_{\RR^{n}}\rho(z)u_{0}(x+z/n)+\sum_{i}f_{i}(x)\int_{\RR^{n}}\rho(z) \tilde u_{i}\left(\tilde\psi^{-1}_{i}(x)+z/n\right)
$$
where $\tilde u_{i}$ is the function defined on $V_{i}\times(-r,r)$ by $\tilde u_{i}(y,s)=u_{0}(\tilde\psi_{i}(y,|s|))$.  Then $u_{0,n}$ approaches $u_{0}$ uniformly on $\bar\Omega$,
and has zero normal derivative for each $n$.

For each $n$ the existence of a regular solution of Equation \eqref{Neumann problem} with initial conditions $u_{0,n}$ is guaranteed by \cite[Lemma 13.21 and Theorem 13.19]{lieb:book}.  Lipschitz bounds for these solutions, depending on $t$ but not on $n$, follow from Theorem \ref{thm:Neumann}.  H\"older estimates on first spatial derivatives are provided by \cite[Theorem 12.3 and Lemma 13.21]{lieb:book}.  Finally, the Schauder estimate of \cite[Theorem 4.31]{lieb:book} gives H\"older estimates on second spatial and first time derivatives, independent of $n$. 

It follows that a sequence converges to a regular limit $u$ which satisfies Equation \eqref{Neumann problem} on $\bar\Omega\times(0,\infty)$, and we argue exactly as in Theorem \ref{thm:existperiodic} to deduce that this converges uniformly to $u_{0}$ as $t\to 0$.  Indeed the barriers defined by \eqref{eq:supersoln} are supersolutions, since at any point $y\in\partial\Omega$ they satisfy 
$D_{{\mathbf n}(y)}w_{a}=v(a)\phi'\langle y-x_{0},{\mathbf n}(y)\rangle\leq 0$ by the convexity of the domain.
\end{proof}

\subsection{The Dirichlet problem}

We consider three cases for the Dirichlet problem:  In the case of a convex domain, an existence result holds in similar generality to Theorems \ref{thm:existperiodic} and \ref{thm:existneumann}.  In the case of `isotropic flows' as discussed in we also obtain an existence result on mean-convex domains, and finally we give an existence result for anisotropic mean curvature flows on domains with non-negative anisotropic mean curvature.
%% convex case, any flow
For the general case we assume the same conditions on $a$, $b$, $\psi$ and $\varphi$ as above.
 
\begin{theorem}\label{thm:existDirichlet1}
Let $\Omega$ be a bounded convex domain in $\RR^{n}$ with $C^{2+\alpha}$ boundary.
Let $u_{0}$ be a continuous function on $\bar\Omega$ which is zero on $\partial\Omega$.  Then there exists a unique $u$ which is continuous on $\bar\Omega\times[0,\infty)$ and is a regular solution of \eqref{eq:Dirichlet} on $\bar\Omega\times(0,\infty)$ with $u(x,0)=u_{0}(x)$ for all $x\in\bar\Omega$.
\end{theorem}

\begin{proof}
There are no substantial differences with the previous case:  In the mollification process to produce smooth initial data $u_{0,n}$ we extend the initial data to be odd rather than even on the sets $V_{i}\times(-r,r)$;  \cite[Theorem 12.16]{lieb:book} is used to produce solutions for each of the initial data $u_{0,n}$, and Lipschitz estimates for positive times follow from Theorem \ref{thm:Dirichletconvex}.  For positive times uniform bounds on the H\"older continuity of first spatial derivatives follow from \cite[Theorem 12.5 and Theorem 12.3]{lieb:book}, and H\"older continuity of second spatial derivatives from \cite[Theorem 4.28]{lieb:book}. This guarantees the existence of a regular limit $u$ which satisfies equation \eqref{eq:Dirichlet} on $\bar\Omega\times(0,\infty)$.  The proof that $u(.,t)$ converges uniformly to $u_{0}$ as $t\to 0$ is identical to that in the previous results, and the same barriers are supersolutions since the definition of the modulus of continuity $\omega$ implies that the supersolutions constructed are non-negative on $\partial\Omega$.
\end{proof}

\begin{theorem}\label{thm:existDirichlet2}
Let $\Omega$ be a bounded domain with $C^{2+\alpha}$ boundary of non-negative mean curvature.  Let $\alpha: [0,\infty)\to(0,\infty)$ and $\beta: [0,\infty)\to(0,\infty)$ be locally Lipschitz functions with $\alpha(0)=\beta(0)$, and
suppose that $\psi$ and $\varphi$ are as in Theorem \ref{Theorem 4} with $|\varphi'|$ bounded for $t>0$.  Then for any
$u_{0}\in C^{0}(\bar\Omega)$ with $u=0$ on $\partial\Omega$ there exists a unique regular solution $u$ of \eqref{dirichlet problem} on $\bar\Omega\times(0,\infty)$ which is continuous on $\bar\Omega\times[0,\infty)$ with $u(x,0)=u_{0}(x)$ for all $x\in\bar\Omega$.
\end{theorem}

The proof of this result is identical to that for Theorem \ref{thm:existDirichlet1}, except that the Lipschitz estimates use Theorem \ref{Theorem 4}.  The assumption on $\alpha$ and $\beta$ guarantees that the coefficients of Equation \eqref{dirichlet problem} are locally Lipschitz functions of $Du$.

\begin{theorem}\label{thm:existDirichlet3}
Let $m$ and $F$ be as described in Equation \eqref{eq:AMCFcoeff}.  Let $\Omega$ be a bounded domain with $C^{2+\alpha}$ boundary which has non-negative anisotropic mean curvature as defined by $F$.  Then for any initial data $u_{0}\in C^{0}(\bar\Omega)$ with $u=0$ on $\partial\Omega$ there exists a unique smooth solution $u$ of \eqref{eq:AMCF} and \eqref{eq:AMCFcoeff} with Dirichlet boundary condition on $\bar\Omega\times(0,\infty)$ which is continuous on $\bar\Omega\times[0,\infty)$ and satisfies $u(x,0)=u_{0}(x)$ for all $x\in\bar\Omega$.
\end{theorem}

Again, the only difference is that the Lipschitz estimates use Theorem \ref{thm:AMCFDirichlet}.

\subsection{Entire initial value problems}

The entire initial value problem can be approached by taking a limit of periodic or Neumann problems.  In the case of initial data of bounded oscillation, we can produce solutions for a very wide class of equations, of similar generality to Theorem \ref{thm:existperiodic}.  We also extend this to a natural class of initial data of linear growth.   In this setting the question of uniqueness is more subtle and we do not address it.

\begin{theorem}\label{thm:entireexist}
Let $u_{0}$ be a continuous function on $\RR^{n}$ with modulus of continuity $\psi(z)$ which is positive and concave on $(0,\infty)$.  Let $a^{ij}$ be locally Lipschitz, and let $b$ be locally H\"older continuous.
Suppose there exists $\varphi$ satisfying \eqref{conditions and equation for varphi} with $\psi$ as above and $\alpha$ defined according to \eqref{defn of alpha}, and with $|\varphi'|$ bounded for $t>0$.
Then there exists a regular solution $u$ to Equation \eqref{evolution equation} on $\RR^n\times(0,\infty)$ which is continuous on $\RR^n\times[0,\infty)$ (with modulus of continuity $\varphi(.,t)$ for each $t\geq 0$) and satisfies $u(x,0)=u_{0}(x)$ for all $x\in\RR^{n}$.
\end{theorem}

In particular, the conditions are satisfied with $\psi(z)=C_1+C_2z$ for any $C_1,C_2\geq 0$ if $\alpha(p)\geq A|p|^{-2}$ for $p$ sufficiently large.  If $\alpha(p)\geq A|p|^{-\gamma}$ for $p$ large with $\gamma>2$ then the result holds with $\psi= C_1z^\beta+C_2z$ for any $C_1>0$ and $C_2\geq 0$, with $\frac{\gamma-2}{\gamma}<\beta\leq 1$ (compare \cite[Example 6.2]{ac:gradest1D}).

\begin{proof}
We construct a solution by taking a limit of solutions $u_{R}$ to the Neumann problem on $B_{R}(0)\times[0,\infty)$ with initial data given by the restriction of $u_{0}$ to $B_{R}(0)$.  Theorem \ref{thm:existneumann} gives the existence of these, and Theorem \ref{thm:Neumann} gives a Global Lipschitz bound on $u_R$, independent of $u_R$.  It follows that the equation is uniformly parabolic on $B_R\times[t_0,\infty)$ for each $t_0$, so bounds for H\"older norms of first spatial derivatives hold on $B_{R/2}\times[2t_0,\infty)$, independent of $R$ \cite[Theorem 12.3]{lieb:book}.  Second spatial derivatives then also satisfy H\"older estimates  independent of $R$, on $B_{R/4}\times[4t_0,\infty)$ \cite[Theorem 4.31]{lieb:book}.  It follows that a subsequence converges (with local uniform convergence of spatial derivatives to second order and first time derivative) to a regular solution $u$ of \eqref{evolution equation} on $\RR^n\times(0,\infty)$.

We prove that $u(.,t)$ converges locally uniformly to $u_0$ as $R\to\infty$.

First, If $u_0$ is uniformly continuous (in particular, in any case where $\lim_{z\to 0}\psi(z)=0$) then the argument is exactly as in the proof of Theorem \ref{thm:existneumann}, and we deduce global uniform convergence of $u(.,t)$ to $u_0$.  
 
 In the general case the argument is a slight modification of that used previously. 
 For each $R>0$ and $\varepsilon>0$ there exists $\delta(R,\varepsilon)>0$ (non-decreasing in $\varepsilon$) such that $|u(y)-u(x)|<\varepsilon$ for all $x\in \bar B_R(0)$ and $y\in\RR^n$ with $|y-x|<\delta(R,\varepsilon)$.  Define $v_R(\varepsilon)=\frac{\psi(\delta(\varepsilon))}{\delta(\varepsilon)}$.  Then by concavity of $\psi$, $v_R(\varepsilon)$ is non-increasing in $\varepsilon$, and we have $|u(y)-u(x)|\leq \varepsilon+v_R(\varepsilon)|y-x|$ for all $x\in B_R(0)$ and $y\in\RR^n$.  For fixed $R_0$, the supersolutions and subsolutions constructed according to \eqref{eq:supersoln} therefore give bounds on $\sup_{x\in B_{R_0}(0)}|u_R(x,t)-u_R(x,0)|$ independent of $R$ for $R$ large, establishing that $u_R(.,t)$ converges uniformly to $u_0$ on $B_{R_0}(0)$ as $R\to\infty$, for any $R_0$.

\end{proof}

\bibliographystyle{amsplain}

\end{document}